\documentclass[letterpaper, 10 pt, conference]{ieeeconf}  

\IEEEoverridecommandlockouts                              
\usepackage[english]{babel}
\usepackage{bold-extra}
\usepackage{amsmath,amssymb}
\usepackage[utf8]{inputenc} 
\usepackage[T1]{fontenc}    
\usepackage{hyperref}       
\usepackage{url}            
\usepackage{booktabs}       
\usepackage{nicefrac}       
\usepackage{microtype}      
\usepackage{xcolor}         
\usepackage{dsfont,mathtools}
\usepackage[linesnumbered,ruled,vlined]{algorithm2e}
\usepackage{subcaption}
\usepackage{graphicx}
\usepackage{optidef}
\usepackage{diagbox}
\usepackage{tikz-cd}
\usepackage{multirow,arydshln} 
\usepackage{optidef}
\usepackage{wrapfig}

\makeatletter
\newcommand{\mytag}[2]{%
  \text{#1}%
  \@bsphack
  \begingroup
    \@onelevel@sanitize\@currentlabelname
    \edef\@currentlabelname{%
      \expandafter\strip@period\@currentlabelname\relax.\relax\@@@%
    }%
    \protected@write\@auxout{}{%
      \string\newlabel{#2}{%
        {#1}%
        {\thepage}%
        {\@currentlabelname}%
        {\@currentHref}{}%
      }%
    }%
  \endgroup
  \@esphack
}
\makeatother

\graphicspath{{}{figure/}}

\newcommand{\pp}{\mathbf{P}}

\newcommand{\bm}{\mathbf{m}}

\newcommand{\bx}{\mathbf{x}}

\newcommand{\calS}{\mathcal{S}}

\newcommand{\by}{\mathbf{y}}

\newcommand{\bv}{\mathbf{v}}

\newcommand{\rel}[1][N]{V^{}_{\mathrm{rel}}(\bm(0),T)}

\newcommand\ba{\mathbf{a}}
\newcommand\bs{\mathbf{s}}

\newcommand\bzero{\mathbf{0}}

\newcommand\calO{\mathcal{O}}
\newcommand\calE{\mathcal{E}}
\newcommand\calON{\mathcal{O}(1/\sqrt{N})}

\newcommand\calA{\mathcal{A}}

\newcommand\calU{\mathcal{U}}

\newcommand\calL{\mathcal{L}}

\newcommand{\expect}[1]{\mathbb{E}\left[#1\right]}

\newcommand{\proba}[1]{\mathbb{P}\left(#1\right)}

\newcommand{\abs}[1]{\left|#1\right|}

\newcommand{\norme}[1]{\left\| #1 \right\|_1}

\newcommand\calUN{\mathcal{U}^{(N)}}
\newcommand\DeltaN{\Delta^{(N)}}

\newcommand\piN{\pi^{N}}

\newcommand\bX{\mathbf{X}}
\newcommand\bU{\mathbf{U}}
\newcommand\bu{\mathbf{u}}
\newcommand\br{\mathbf{r}}
\newcommand\limT{\lim_{T \rightarrow \infty} \frac{1}{T} \sum_{t=0}^{T-1}}
\newcommand\limN{\lim_{N \rightarrow \infty} }

\newcommand\palpha{\mathbf{P}^{\alpha}}
\newcommand\pimax{\pi_{\text{align\&}\ell}}
\newcommand\pidelay{\pi_{\text{delay\&}\ell}}

\newcommand\barU{\overline{U}}
\newcommand\barbU{\overline{\mathbf{U}}}
\newcommand\bh{\mathbf{h}}
\newcommand\bone{\mathbf{1}}

\newcommand\calC{\mathcal{C}}
\newcommand\pil{\pi_{\ell}}

\newtheorem{defi}{\textbf{Definition}}
\newtheorem{thm}{\textbf{Theorem}}
\newtheorem{lem}{\textbf{Lemma}}

\title{\LARGE \bf
 An Optimal-Control Approach to Infinite-Horizon Restless Bandits: Achieving Asymptotic Optimality with Minimal Assumptions
}

\author{Chen Yan 
\thanks{Chen Yan is with Statify, Inria, 38334 Saint Ismier, France and Biostatistics and Spatial Processes, INRAE, 84914 Avignon, France.
        {\tt\small chen.yan@inria.fr}}%
}

\begin{document}

\maketitle 

\thispagestyle{empty} 
\pagestyle{empty}

\begin{abstract}
We adopt an optimal-control framework for addressing the undiscounted infinite-horizon discrete-time restless $N$-armed bandit problem. Unlike most studies that rely on constructing 
policies based on the relaxed single-armed Markov Decision Process (MDP), we propose relaxing the entire bandit MDP as an optimal-control problem through the certainty equivalence control 
principle. Our main contribution is demonstrating that the reachability of an optimal stationary state within the optimal-control problem is a sufficient condition for the existence of an 
asymptotically optimal policy. Such a policy can be devised using an "align and steer" strategy. This reachability assumption is less stringent than any prior assumptions imposed on the 
arm-level MDP, notably the unichain condition is no longer needed. Through numerical examples, we show that employing model predictive control for steering generally results in superior 
performance compared to other existing policies. 
\end{abstract}

\section{Introduction}

The Restless Bandit (RB) problem addresses the challenge of optimally allocating limited resources across a set of dynamically evolving alternatives \cite{whittle-restless}. Each 
alternative, or "arm", changes state over time according to a Markov Decision Process (MDP), irrespective of whether it is currently being exploited or not, hence the term "restless". This 
problem encapsulates a broad range of real-world scenarios, from queue management  and sensor scheduling  to wireless communication  and adaptive clinical trials . Despite its theoretical 
and practical significance, finding optimal solutions remains notoriously challenging \cite{Papadimitriou99thecomplexity}, driving ongoing research into efficient heuristics and 
asymptotically optimal policy design \cite{Ve2016.6,Brown2020IndexPA,gast2023linear,hong2023restless}. This paper contributes to this vibrant field by proposing an optimal-control framework 
that offers fresh insights into the asymptotic optimality of policies for the RB problem. 


\textbf{Contributions:} 
\begin{itemize}
  \item We propose a novel approach by relaxing the stochastic bandit problem into a deterministic optimal-control problem, diverging from the conventional strategy of relaxation into a 
      single-armed problem (see Figure~\ref{fig:flowchart}). 
  \item We link asymptotic optimality in the bandit problem to the reachability of an optimal stationary point via feasible control, bypassing the unichain assumption for a broader 
      applicability that includes multichain models.
  \item We propose the "align and steer" strategy for constructing asymptotically optimal policies, assuming reachability. Our numerical studies highlight the superiority of integrating 
      model predictive control within this strategy. 
\end{itemize}


\textbf{Notations:} To differentiate between the single-armed MDP and the $N$-armed bandit MDP, we use the letter $s$ to denote the state of the former, which assumes a finite set of $S$ 
values, and $\bx, \bX$ for the state of the latter, represented as a population vector within the unit simplex $\Delta$ of dimension $S$ upon dividing by $N$. For the bandit-level problem, 
capital letters indicate stochastic systems, lowercase for deterministic, and boldface for vectors, treated as row vectors. The subset $\DeltaN$ of $\Delta$ consists of points whose 
coordinates are multiples of $1/N$. Vector inequality $\bx \ge \by$ are defined componentwise.  We use \emph{control rule} $\pi$ for deterministic optimal-control problems and \emph{policy} 
$\piN$ for stochastic $N$-armed bandit MDPs. Control mappings are denoted as $\pi(\bx) = \bu$, with $\bx^\pi(t)$ (resp. $\bu^\pi(t)$) representing the state (resp. control) after applying 
$\pi$ over $t$ steps on an initial state $\bx(0)$.

\section{Problem Setup and Literature Review}  \label{sec:problem-formulation}

\subsection{Model Description}

Consider the undiscounted infinite-horizon discrete-time Restless Bandit (RB) problem with $N$ homogenous arms. Each arm itself is a Markov Decision Process (MDP) with state space $\calS := 
\{1,2,\dots,S \}$ and action space $\calA := \{0,1\}$. There is a budget constraint requiring that at each time step, exactly $\alpha N$ arms can take action $1$, with $0 < \alpha < 1$. For 
simplicity we assume that $\alpha N$ is always an integer. The state space of the $N$-armed bandit is therefore $\calS^{N}$ and the action space is a subset of $\calA^{N}$. The arms are 
weakly-coupled, in the sense that they are only linked through the budget constraint, i.e. for a given feasible action $\ba \in \calA^N$, the bandit transitions from a state $\bs \in 
\calS^N$ to state $\bs' \in \calS^N$ with probability $\proba{\bs' \mid \bs,\ba} = \prod_{n=1}^{N} \proba{ s'_n \mid s_n,a_n} = \prod_{n=1}^{N} P^{a_n}_{s_n, s'_n}$, where for each action 
$a_n = a \in \calA$, the matrix $\pp^a$ is a probability transition matrix of dimension $S \times S$. Upon choosing an action $\ba$ in state $\bs$, we receive an instant-reward 
$\sum_{n=1}^{N} r_{s_n}^{a_n}$, where $r_{s}^{a} \in \mathbb{R}$ depends on the state $s$ and action $a$. 

A \emph{Markovian} policy $\piN$ for the $N$-armed problem chooses at each time $t$ a feasible action $\ba(t)$ based solely on the current state $\bs(t)$. It is \emph{stationary} if in 
addition it does not depend on $t$. Our goal is to maximize the long-term average expected reward from all $N$ arms across all stationary policies, facing an exponentially large state and 
action space as $N$ increases. \footnote{ In contrast to stochastic and adversarial bandits, where the model is not fully known and the emphasis is on minimizing regret compared to a 
hindsight optimal \cite{lattimore2020bandit}, the current Markovian bandit setting assumes all problem parameters and the system states are known, focusing on the design of efficient and 
effective algorithms. } Formally, this bandit MDP with a given initial state $\bs(0)$ is formulated as: 
\begin{alignat}{2}
    & \underset{\piN }{\text{max}} \quad && V_{\piN}(\bs(0)) := \limT \frac{1}{N} \mathbb{E}_{\piN} \left[ \sum_{n=1}^{N} r_{s_n(t)}^{a_n(t)} \right] 
    \label{eq:problem-formulation-all-arm} \\
    & \text{s.t.} \quad && \begin{aligned}[t]
    \proba{\bs(t+1) \mid \bs(t),\ba(t)} = \prod_{n=1}^{N} P^{a_n(t)}_{s_n(t), s_n(t+1)}, 
    \end{aligned}  \label{eq:Markovian-evolution} \\
    & && \mathbf{a}(t) \cdot \mathbf{1}^\top = \alpha N, \quad \mathbf{a}(t) \in \{0,1\}^N  \quad \forall t \ge 0. \label{eq:constraint-all-t} 
\end{alignat}

The difficulty of the RB problem is that the $N$ arms are coupled by the constraints in Equation~\eqref{eq:constraint-all-t}, and the conventional approach begins with relaxing these 
constraints in Equation~\eqref{eq:constraint-all-t}, which must be met at every time $t$ with probability one, to a single time-averaged constraint in expectation: $\limT \mathbb{E}_{\piN} 
\left[ \ba(t) \cdot \bone^\top \right] = \alpha N$. This effectively decompose the $N$-armed problem into $N$ independent single-armed problem, each having the following form in relation to 
a single-armed policy $\bar{\pi}$ and an initial arm state $s(0)$: 
\begin{alignat}{2}
    & \underset{\bar{\pi} }{\text{max}} \quad && V_{\bar{\pi}}(s(0)) := \limT \mathbb{E}_{\bar{\pi}} \left[ r_{s(t)}^{a(t)} \right] \label{eq:problem-formulation-single-arm} \\
    & \text{s.t.} \quad && \mathbb{P}(s(t+1)  \mid s(t), a(t)) = P_{s(t), s(t+1)}^{a(t)}, \quad \forall t \ge 0  \notag \\
    & && \limT \mathbb{E}_{\bar{\pi}} \left[ a(t) \right] = \alpha.  \notag 
\end{alignat}
This single-armed problem can be equivalently formulated using state-action frequency variables, see \cite[Section 8.9.2]{Puterman:1994:MDP:528623} and 
Problem~\eqref{eq:problem-formulation-stationary-multichain} below. We denote by $\bar{\pi}^*$ as one such optimal single-armed policy with an optimal value $V_{\bar{\pi}^*}(s(0))$. Note 
that the initial arm state $s(0) \in \calS$ can be extended to a probability distribution in $\calS$, represented by a point $\bx(0) \in \Delta$. Under the \emph{unichain} assumption, the 
optimal value $V_{\bar{\pi}^*}(\bx(0))$ of Problem~\eqref{eq:problem-formulation-single-arm} is independent of the initial arm state distribution $\bx(0)$ \cite{Puterman:1994:MDP:528623}.

\subsection{The Approach via Optimal-Control}

In this paper, our approach is to approximate Problem~\eqref{eq:problem-formulation-all-arm} with an optimal-control problem via the Certainty Equivalence Control (CEC) principle 
(\cite[Chapter 6]{bertsekas2012dynamic}). Throughout this paper, we will not make the blanket assumption that the single-armed MDP is unichain. 

\subsubsection{Arm States Concatenation and the CEC Problem}

Given that the $N$ arms are homogeneous, representing the bandit state through the concatenation of arm states can significantly simplify subsequent analysis. To achieve this, denote by 
$\bX \in \DeltaN$ where $X_s$ is the fraction of arms in state $s \in \calS$, normalized by division by $N$. A similar notation goes for the control $\bU$, so that $U_s$ is the fraction of 
arms in state $s$ taking action $1$ under the control $\bU$. 

Using this arm states concatenation, the Markovian evolution of the bandit state in Equation~\eqref{eq:Markovian-evolution} can be expressed as $\bX(t+1) \overset{d}{=} \phi(\bX(t), \bU(t)) 
+ \mathcal{E}(\bX(t), \bU(t))$, where $\phi(\cdot, \cdot)$ is the deterministic \emph{linear} function: 
\begin{equation}  \label{eq:deterministic-transition}
  \phi(\bX, \bU) := (\bX - \bU) \cdot \pp^0 + \bU \cdot \pp^1,
\end{equation}
and $\mathcal{E}(\cdot, \cdot)$ is a Markovian random vector, whose properties are summarized in the following lemma, with a proof utilizing standard probability techniques available in 
\cite[Lemma 1]{gast2023linear}: 

\begin{lem}[\cite{gast2023linear}]  \label{lem:Markovian-transition-analysis} 
The random vector $\calE(\bX(t),\bU(t)) \overset{d}{=} \bX(t+1) - \phi(\bX(t),\bU(t))$ verifies:
  \begin{align*}
   & \expect{\calE(\bX,\bU) \mid \bX,\bU } = \bzero; \ \expect{\norme{\calE(\bX,\bU)} \mid \bX,\bU } \le \sqrt{S}/\sqrt{N}; \\
   & \proba{\norme{\calE(\bX,\bU)} \ge \xi \mid \bX,\bU } \le 2S \cdot e^{-2N \xi^2/S^2}.    
  \end{align*}
\end{lem}

Given a state $\bx \in \Delta$, define the following two control sets:
\begin{align*}
  \calU(\bx) & := \left\{ \bu \mid  \bu \cdot \bone^\top = \alpha \text{ and } \bzero \le \bu \le \bx \right\}; \\
  \calUN(\bx) & := \left\{ \bu \mid  \bu \in \calU(\bx) \text{ and } N \cdot \bu \text{ is an integer vector} \right\};
\end{align*}
as well as the linear instant-reward function:
\begin{equation*} 
  R(\bx,\bu) := (\bx-\bu) \cdot (\br^0)^\top + \bu \cdot (\br^1)^\top.
\end{equation*}
Note that $\calU(\bx)$ is always non-empty, and $\calUN(\bx)$ is non-empty if $\bx \in \DeltaN$. A \emph{feasible control} $\pi$ is a map from $\bx \in \Delta$ to $\calU(\bx)$, while a 
\emph{feasible policy} $\piN$ maps $\bx$ into $\calUN(\bx)$. An equivalent formulation of Problem~\eqref{eq:problem-formulation-all-arm} using arm states concatenation (where $\bs(0)$ 
yields $\bx(0)$) is: 
\begin{alignat}{2}
    & \underset{\piN}{\text{max}} \quad && V_{\piN}(\bx(0)) := \limT \mathbb{E}_{\piN} \left[ R(\mathbf{X}(t),\mathbf{U}(t)) \right] 
    \label{eq:problem-formulation-population-representation-infinite-horizon} \\
    & \text{s.t.} \quad && \mathbf{X}(t+1) \overset{d}{=} \phi(\mathbf{X}(t), \mathbf{U}(t)) + \mathcal{E}(\mathbf{X}(t), \mathbf{U}(t)),  \notag \\ 
    & && \mathbf{U}(t) \in \calUN(\mathbf{X}(t)) \ a.s., \quad \forall t \ge 0. \notag  
\end{alignat}
The two requirements below Equation~\eqref{eq:problem-formulation-population-representation-infinite-horizon} result from arm states concatenation of 
Equations~\eqref{eq:Markovian-evolution} and \eqref{eq:constraint-all-t}, respectively. 

We now link Problem~\eqref{eq:problem-formulation-population-representation-infinite-horizon} to its corresponding CEC problem, where the uncertainties $\mathcal{E}(\cdot, \cdot)$ are 
assumed to be identically zero. Specifically, the CEC problem is defined as a maximization task over all stationary control rules $\pi$, described as follows: 
\begin{alignat}{2}
    & \underset{\pi}{\text{max}} \quad && V_{\pi}(\bx(0)) := \limT R(\mathbf{x}(t),\mathbf{u}(t)) \label{eq:problem-CEC-infinite-horizon} \\
    & \text{s.t.} \quad && \mathbf{x}(t+1) =  \phi \left(\bx(t), \bu(t) \right), \notag \\
    & && \mathbf{u}(t) \in \mathcal{U}(\mathbf{x}(t)), \quad \forall t \ge 0. \notag 
\end{alignat}
Bearing its resemblance to Problem~\eqref{eq:problem-formulation-population-representation-infinite-horizon}, the above Problem~\eqref{eq:problem-CEC-infinite-horizon} is now deterministic 
with \emph{uncountable} state and action space. In contrast to the single-armed MDP Problem~\eqref{eq:problem-formulation-single-arm} that relaxes the constraints 
\eqref{eq:constraint-all-t} into a \emph{single} time-averaged expectation constraint, the CEC problem relaxes these constraints into expectation constraints at \emph{every} time step, 
represented by $\bu(t) \cdot \bone^\top = \alpha$ for all $t \ge 0$. It is a \emph{linear} control problem where the set of feasible controls depends on the state. Note that there are two 
distinct paths that naturally lead to consider Problem~\eqref{eq:problem-CEC-infinite-horizon}: the first involves taking the large $N$ limit in 
Problem~\eqref{eq:problem-formulation-population-representation-infinite-horizon} and referring to Lemma~\ref{lem:Markovian-transition-analysis} as we previously discussed; the second 
entails taking the large $T$ limit from the finite-horizon RB relaxation to a linear program, as considered in various works \cite{Brown2020IndexPA,gast2023linear}. 

As Problem~\eqref{eq:problem-formulation-population-representation-infinite-horizon} is generally intractable \cite{Papadimitriou99thecomplexity}, our strategy employs a stationary control 
rule $\pi$ that optimally solves the more tractable Problem~\eqref{eq:problem-CEC-infinite-horizon}. From this, we construct an induced policy $\piN$ that matches as much as possible to 
$\pi$. This is made precise in the following definition: 

\begin{defi}(Induced Policy $\piN$ from Control Rule $\pi$)  \label{def:rounding-policy}
  For a feasible stationary control rule $\pi$ of Problem~\eqref{eq:problem-CEC-infinite-horizon}, a corresponding induced policy $\piN$ for 
  Problem~\eqref{eq:problem-formulation-population-representation-infinite-horizon} is defined as any stationary policy such that for an input state $\bX \in \DeltaN$ with 
  $\barbU = \pi(\bX)$, it outputs a control $\piN(\bX) = \bU \in \calUN(\bX)$ satisfying $\norme{\bU - \barbU} \le S/N$. \footnote{The requirement $\norme{\bU - \barbU} \le S/N$ can be 
  met by setting $U_s := (\lfloor N \cdot \barU_s \rfloor + Z_s)/N$, with $Z_s$ appropriately chosen to take values of either $0$ or $1$ so as to meet the budget constraint. Furthermore, a 
  more sophisticated randomized rounding strategy \cite[Section 2.3]{gast2023linear} can ensure $\expect{\bU} = \barbU$. }
\end{defi}

The general observation from Lemma~\ref{lem:Markovian-transition-analysis} is that if $\pi$ is optimal, then the induced policy $\piN$ as in Definition~\ref{def:rounding-policy} should also 
be close to optimal for large values of $N$. This will be precisely formulated in Theorem~\ref{thm:cv-rate} below. 

\subsubsection{Stationary Problems}

By definition, a \emph{stationary point} $(\bx_e, \bu_e)$ of Problem~\eqref{eq:problem-CEC-infinite-horizon} is one such that $\bu_e \in \calU(\bx_e)$ and $\bx_e = \phi(\bx_e,\bu_e)$. A 
stationary point $(\bx^*, \bu^*)$ is optimal if it solves the corresponding static problem with \eqref{eq:problem-CEC-infinite-horizon}. This is what we refer to as the \emph{conventional} 
static problem. In this paper, however, in order to also take into account multichain models, we shall consider the refined static problem with optimal value denoted as $V^*_e(\bx(0))$, 
following \cite{altman1995linear,hordijk1984constrained}:  
\begin{alignat}{2}
    & \underset{\mathbf{x}, \mathbf{u}, \mathbf{h}^0, \mathbf{h}^1}{\text{max}} \quad && V_e(\bx(0)) := R(\mathbf{x},\mathbf{u}) \label{eq:problem-formulation-stationary-multichain} \\
    & \text{s.t.} \quad && \mathbf{x} =  \phi \left( \bx,\bu \right), \notag \\
    & && \mathbf{u} \in \mathcal{U}(\mathbf{x}), \notag \\
    & && \mathbf{x} + \mathbf{h}^0 + \mathbf{h}^1 - \mathbf{h}^0 \cdot \mathbf{P}^0 - \mathbf{h}^1 \cdot \mathbf{P}^1 = \mathbf{x}(0), \label{eq:mass-multi} \\
    & && \mathbf{x} \ge \mathbf{0}, \ \mathbf{h}^0 \ge \mathbf{0}, \ \mathbf{h}^1 \ge \mathbf{0}. \notag
\end{alignat}
We recover the conventional static problem if in the above optimization problem Equation~\eqref{eq:mass-multi} is replaced by $\bx \cdot \bone^\top = 1$ and there are no $\bh^0, \bh^1$ 
variables. Problem~\eqref{eq:problem-formulation-stationary-multichain} is a refinement since multiply Equation~\eqref{eq:mass-multi} by $\bone^\top$ on the right gives the relation $\bx 
\cdot \bone^\top = 1$. The additional variables $\bh^0, \bh^1$ appearing in Problem~\eqref{eq:problem-formulation-stationary-multichain} can be interpreted as a deviation measure 
\cite{altman1995linear}. In fact, by \cite[Theorem 10]{hordijk1984constrained}, if the single-armed MDP is unichain, then for any initial condition $\bx(0)$ 
Problem~\eqref{eq:problem-formulation-stationary-multichain} is equivalent to the conventional static problem, so the two formulations make no difference. However, we will illustrate in 
Section~\ref{subsec:discussion-reachability} the necessity of considering the refined static problem for the more general multichain models. 

\subsubsection{Value Comparison and Asymptotic Optimality}

\begin{figure}[htbp]
  \centering
  \includegraphics[width=0.49\textwidth]{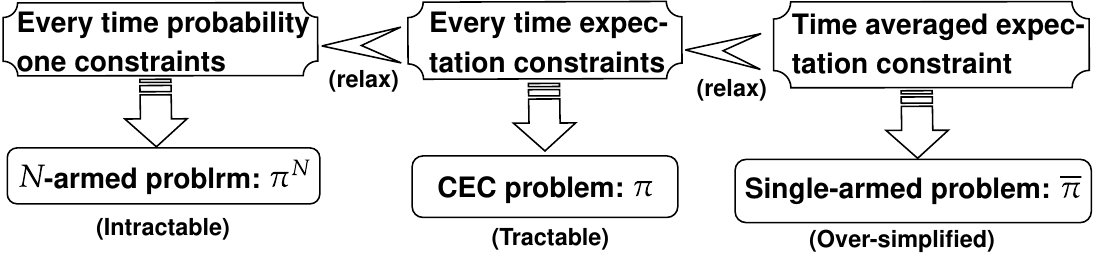}
   \caption{\small Relationship of the three optimization problems}
  \label{fig:flowchart}
\end{figure}
The links between the three major problems considered in this work are summarized in Figure~\ref{fig:flowchart}. The relationships among the values of the various optimization problems 
defined up to this point are as follows: 
\begin{align}  
 & V_{\piN}(\bs(0)) =  V_{\piN}(\bx(0)) \nonumber \\
 & \le V_{\pi}(\bx(0)) \le V_{\bar{\pi}^*}(\bx(0)) = V^*_e(\bx(0)),  \label{eq:relation-of-problem-values}
\end{align}
where the first relationship arises from the concatenation of arm states; the second is established in Lemma~\ref{lem:CEC-is-upperbound} of Section~\ref{appendix:proof-lemmas}, whose proof 
is based on induction on the horizon; the third results from Problem~\eqref{eq:problem-formulation-single-arm} being a more relaxed formulation than 
Problem~\eqref{eq:problem-CEC-infinite-horizon}; and the final relationship can be deduced, for instance, from \cite{hordijk1984constrained}, by treating $\bx-\bu$ and $\bu$ as state-action 
frequency variables. Equation~\eqref{eq:relation-of-problem-values} leads us to propose the following definition: 
\begin{defi}(Optimal Stationary Point and Asymptotic Optimality)  \label{def:asymptotic-optimality} 
For a given initial state $\bx_{\text{init}}$, we call $(\bx^*,\bu^*)$ that solves the refined static Problem~\eqref{eq:problem-formulation-stationary-multichain} with $\bx(0) = 
\bx_{\text{init}}$ an \emph{optimal stationary point}, and $\bx^*$ an optimal stationary state. We call a control rule $\pi$ of Problem~\eqref{eq:problem-CEC-infinite-horizon}  
\emph{averaged-reward optimal} if $V_{\pi}(\bx(0)) = V^*_e(\bx(0))$; the corresponding stationary policy $\piN$ of 
Problem~\eqref{eq:problem-formulation-population-representation-infinite-horizon} in Definition~\ref{def:rounding-policy} is said to be \emph{asymptotically optimal} for the $N$-armed RB 
problem, if it verifies $\limN V_{\piN}(\bx(0)) = V^*_e(\bx(0))$. 
\end{defi}

\subsection{Comparison with Related Works}

Existing literature on infinite-horizon RB problems typically constructs policies $\piN$ based on the single-armed problem in Figure~\ref{fig:flowchart}. We have demonstrated that the 
latter essentially conveys the same information as an optimal stationary point within this control problem framework, thus revealing that achieving asymptotic optimality necessitates 
additional model assumptions. 

Firstly, traditional methods assume the unichain condition for the single-armed MDP. Additionally, specific policies like the Whittle index policy \cite{WeberWeiss1990} and the fluid 
priority policy \cite{Ve2016.6} necessitate their induced dynamical systems conforming to the Global Attractor Property. Furthermore, to affirm the exponential convergence rate, 
\cite{gast2023linear} introduces a more stringent Uniform Global Attractor Property (UGAP). These assumptions on global dynamical system behavior are theoretically challenging to verify. In 
response, \cite{hong2023restless} proposes a more easily verifiable Synchronization Assumption for the optimal single-armed policy $\bar{\pi}^*$, achieving $\calON$ asymptotic optimality 
with the Follow-the-Virtual-Advice (FTVA) policy. Extending this approach, \cite{hong2024unichain} further simplifies the criteria, showing that the unichain and aperiodic assumptions on 
$\bar{\pi}^*$ are sufficient for their ID / Focus Set policies. Lemma~\ref{lem:property-of-pi0} will demonstrate that this sufficient condition implies our reachability condition in 
Definition~\ref{def:reachability}. On the other hand, without the unichain assumption, a stationary optimal single-arm policy may not even exist \cite{hordijk1984constrained}, yet our 
findings still enable the construction of asymptotically optimal policies, see Section~\ref{subsec:discussion-reachability}. It is important to recognize that verifying an MDP's unichain 
assumption is algorithmically demanding \cite{kallenberg2002classification}, making it practical to initially consider the problem as multichain \cite[Chapter 9]{Puterman:1994:MDP:528623}. 

A comprehensive comparison of various existing policies on the undiscounted infinite-horizon discrete-time RB problem is summarized in the table below. 
\begin{table}[ht]
\centering
\begin{tabular}{|p{3.05cm}|p{3.19cm}|p{1.1cm}|}
\hline
Policies           & Assumptions            & Rate  \\ \hline
Whittle / LP-Priority \cite{gast2023linear}    & UGAP \&Regular\& Unichain & $e^{-\Omega(N)}$  \\ \hline
FTVA \cite{hong2023restless}           & Synchronization \& Unichain      & $\calON$      \\ \hline
ID / Focus-Set \cite{hong2024unichain}       & Aperiodic \& Unichain  & $\calON$   \\ \hline
Align and Steer [this work]  & Weaker than Aperiodic   & $o(1)$      \\ \hline
\end{tabular}
\end{table}

\textbf{Roadmap:} Previous methods fall short in the more general framework, mainly due to their reliance on the overly simplified static problem. We overcome this issue by focusing on the 
dynamical yet still tractable CEC problem. Theorem~\ref{thm:cv-rate} demonstrates that for the induced policy $\piN$ to achieve asymptotic optimality, the corresponding control rule $\pi$ 
must be average-reward optimal, and additionally, a bias-related term (Equation~\eqref{eq:bias}) needs to be a continuous function. Theorem~\ref{thm:existence-of-good-control-rule} 
demonstrates that this can be achieved through the "align and steer" strategy (Algorithm~\ref{algo:maximum-alignment-control}) with a linear control for steering, provided that a mild 
reachability condition, as detailed in Definition~\ref{def:reachability}, is met by the model.

\section{Reachability and Asymptotic Optimality}  \label{sec:main-results}

Throughout this section, we fix an initial condition $\bx_{\text{init}} \in \Delta$ and a corresponding optimal stationary point $(\bx^*, \bu^*)$ of 
Problem~\eqref{eq:problem-formulation-stationary-multichain}. Denote by $\calS^+ := \left\{ s \in \calS \mid x^*_s > 0 \right\}$. It is important to remember that, in the case of multichain 
models, all quantities we deduce are dependent on the initial state. We will make this dependence explicit whenever possible.

\subsection{The Effective Control Rules}  \label{subsec:effective-control-rule}

In order to formulate the effectiveness of a control beyond average-reward optimality, we define, for a control rule $\pi$, the following possibly unbounded functions for all $\bx = 
\bx(0)$: 
\begin{equation}\label{eq:bias}
  G^{\pi}(\bx) := \sum_{t=0}^{\infty} (\bx^{\pi}(t),\bu^{\pi}(t)) - (\bx^*, \bu^*),   
\end{equation} 
where we recall that $(\bx^*, \bu^*)$ is an optimal stationary point by solving Problem~\eqref{eq:problem-formulation-stationary-multichain} with initial state $\bx$. We define a stationary 
control rule $\pi$ for Problem~\eqref{eq:problem-CEC-infinite-horizon} as \emph{effective} if $G^{\pi}(\bx)$ is a continuous function over $\bx \in \Delta$ under the $\mathcal{L}^1$-norm. 
This implies that $G^{\pi}(\cdot)$ is bounded given the compactness of $\Delta$. The concept and justification for an effective control rule are encapsulated in the subsequent theorem: 

\begin{thm}(Effective Control Rule Leads to Asymptotically Optimal Policy) \label{thm:cv-rate}
Fix an initial state $\bx_{\text{init}}$. For a stationary control rule $\pi$ of Problem~\eqref{eq:problem-CEC-infinite-horizon} with an optimal stationary point $(\bx^*, \bu^*)$ and 
optimal value $V^*_e(\bx(0))$ defined in Problem~\eqref{eq:problem-formulation-stationary-multichain} with $\bx(0) = \bx_{\text{init}}$, consider the function $G^{\pi}(\bx)$ defined in 
Equation~\eqref{eq:bias}. If $G^{\pi}(\bx)$ is a continuous function over $\bx \in \Delta$ (under the $\mathcal{L}^1$-norm). Then the induced stationary policy $\piN$ for 
Problem~\eqref{eq:problem-formulation-population-representation-infinite-horizon} in Definition~\ref{def:rounding-policy} is asymptotically optimal: $\limN V_{\piN}(\bx(0)) = 
V^*_e(\bx(0))$.  
\end{thm}

The proof of Theorem~\ref{thm:cv-rate} employs the standard Stein's method (\cite{gast:hal-01622054}) and is deferred to Section~\ref{appendix:proof-theorem-1}. Informally, for a control 
rule $\pi$ to be averaged-reward optimal, the state-control pairs $(\bx^{\pi}(t),\bu^{\pi}(t))$ need to converge to $(\bx^*, \bu^*)$ independently of the initial state. To establish and 
study refined notions of optimality beyond the average-reward criterion, particularly for comparison with the stochastic $N$-armed problem, we must consider a function of the type 
$G^{\pi}(\bx)$. We refer to \cite{carlson2012infinite} for an in-depth discussion of various optimality criteria in this context. 

It is important to note that with further regularity of the function $G^\pi(\bx)$ in the vicinity of $\bx^*$, i.e. Lipschitz-continuity or $\calC^1$-smoothness, the convergence rate of 
$V_{\piN}(\bx(0))$ towards $V^*_e(\bx(0))$ can be determined. Similar ideas have been explored in prior research, including \cite{gast:hal-01622054,gast2023exponential}. The primary 
challenge is determining whether such an effective control rule can be established and the methodology for its construction, which we aim to explore subsequently.

\subsection{The Align and Steer Policy}  \label{subsec:align-and-steer-policy}

Our idea of constructing an effective control rule can be summarized as "align and steer", which is based on the following observation from the linearity nature of 
Problem~\eqref{eq:problem-CEC-infinite-horizon}: If we decompose $\bx \in \Delta$ into a sum of two parts $\bx = \bv_{\text{align}} + \bv_{\text{steer}}$ with 
$\bv_{\text{align}},\bv_{\text{steer}} \ge \bzero$, then the normalized vectors $ \bx_{\text{align}} := \bv_{\text{align}} / \norme{\bv_{\text{align}}}$ and $ \bx_{\text{steer}} := 
\bv_{\text{steer}} / \norme{\bv_{\text{steer}}}$ again belong to the simplex $\Delta$. Now take $\bu_{\text{align}} \in \calU(\bx_{\text{align}})$ and $\bu_{\text{steer}} \in 
\calU(\bx_{\text{steer}})$ as feasible controls for $\bx_{\text{align}}$ and $\bx_{\text{steer}}$ respectively. The linear combination 
\begin{equation*}
  \norme{\bv_{\text{align}}} \cdot \bu_{\text{align}} + \norme{\bv_{\text{steer}}} \cdot \bu_{\text{steer}}
\end{equation*}
turns out to be a feasible control for $\bx$. The key is to split $\bx$ so that $\bv_{\text{align}}$ is \emph{collinear} with $\bx^*$ and possesses the \emph{maximum} possible 
$\calL^1$-norm. This enables the choice of $\bu_{\text{align}}$ as $\bu^*$ with a maximum alignment (refer to Equation~\eqref{eq:align-and-steer-control} below). 
  
\begin{defi}(Maximum Alignment Coefficient with $\bx^*$) \label{def:maximum-alignment-coef}
For $\bx \in \Delta$, we call the real constant 
\begin{equation}  \label{eq:maximum-alignment-coef}
  \delta(\bx) := \max \{ \delta \ge 0 \mid \bx \ge \delta \cdot \bx^*  \}
\end{equation} 
the maximum alignment coefficient of $\bx$ with the target $\bx^*$. 
\end{defi} 

From this definition, it follows that $0 \le \delta(\bx) \le 1$, with $\delta(\bx) = 0$ if and only if there exists an arm state $s \in \calS^+ = \{ s \in \calS \mid x^*_s > 0 \}$ with $x_s 
= 0$; and $\delta(\bx) = 1$ occurs if and only if $\bx = \bx^*$. For any $\bx \in \Delta$, it can be expressed as $\bx = \delta(\bx) \cdot \bx^* + \bx - \delta(\bx) \cdot \bx^*$. Here, the 
component $\bv_{\text{align}} = \delta(\bx) \cdot \bx^*$ represents the mass in $\bx$ already in alignment with stationarity, whereas $\bv_{\text{steer}} = \bx - \delta(\bx) \cdot \bx^*$ 
requires the application of a specific control, $\pi_{\text{steer}}$, designed to steer the remaining mass into $\calS^+$ for subsequent alignment. This concept is further explored in the 
following subsection. 

Now assume that a certain feasible $\pi_{\text{steer}}$ has been specified. The corresponding \emph{align and steer} control rule $\pi_{\text{align\&steer}}: \bx \mapsto \calU(\bx)$ is 
defined as: 
\begin{align}  
 & \pi_{\text{align\&steer}}(\bx) :=  \nonumber \\ 
 & \delta(\bx) \cdot \bu^* + (1 - \delta(\bx)) \cdot \pi_{\text{steer}} \left(  \frac{\bx - \delta(\bx) \cdot \bx^*}{1 - \delta(\bx)} \right).  \label{eq:align-and-steer-control}
\end{align}
We emphasize that $\delta(\bx)$ plays a crucial role in ensuring that $(\bx - \delta(\bx) \cdot \bx^*)/(1 - \delta(\bx))$ is a well-defined state vector in the simplex $\Delta$. Algorithm 
\ref{algo:maximum-alignment-control} describes the induced align and steer policy $\piN_{\text{align\&steer}}: \bx \mapsto \calUN(\bx)$ in detail. 

\begin{algorithm}
  \SetAlgoLined
  \SetKwInput{KwInput}{Input}
  \KwInput{A feasible steering control $\pi_{\text{steer}}$ of Problem~\eqref{eq:problem-CEC-infinite-horizon}; An initial state $\bX_{\text{init}}$ for the $N$-armed bandit 
  Problem~\eqref{eq:problem-formulation-population-representation-infinite-horizon}.}
  Solve the static Problem~\eqref{eq:problem-formulation-stationary-multichain} with initial state $\bx(0) = \bX_{\text{init}}$ to obtain an optimal stationary state $\bx^*$\;
  Set $\bX_{\text{current}} := \bX_{\text{init}}$ \;
  \For{$t = 0,1,2,\dots$}{
   Set $\barbU_{\text{current}} := \pi_{\text{align\&steer}}(\bX_{\text{current}})$ from Equation~\eqref{eq:align-and-steer-control} \;
   Compute a control $\bU_{\text{current}}$ with inputs $\bX_{\text{current}}$ and $\barbU_{\text{current}}$ as outlined in Definition~\ref{def:rounding-policy} \;
   Apply $\bU_{\text{current}}$ on $\bX_{\text{current}}$ and advance to the next state $\bX_{\text{next}}$, then set $\bX_{\text{current}} := \bX_{\text{next}}$ \;
  }
\caption{The align and steer policy $\piN_{\text{align\&steer}}$.}
  \label{algo:maximum-alignment-control}
\end{algorithm}

An advantage of the align and steer approach is the considerable flexibility in selecting the appropriate steering control, $\pi_{\text{steer}}$, in Algorithm 
\ref{algo:maximum-alignment-control}. Throughout the rest of this section, we focus on the linear steering control $\pil(\bx) = \alpha \cdot \bx$. Our objective is to introduce a mild 
reachability assumption, under which $\piN_{\text{align\&}\ell}$ is theoretically established to be asymptotically optimal. Subsequently, in the next Section~\ref{sec:numerical-example}, we 
adopt Model Predictive Control (MPC) as the steering control and conduct numerical studies on $\piN_{\text{align\&MPC}}$. 

\subsection{Reachability and a Linear Control Rule}  \label{subsec:reachability-and-a-linear-control-rule}

To introduce the key concept of reachability, we derive two observations from the construction in Equation~\eqref{eq:align-and-steer-control}. First, for any given $\bx \in \Delta$ with 
$\delta_0 := \delta(\bx)$ and any $t \ge 0$, the maximum alignment coefficient of $\bx^{\pi_{\text{align\&steer}}}(t)$, after applying $\pi_{\text{align\&steer}}$ on $\bx$ for $t$ steps, is 
at least as large as $\delta_0$. Second, defining $\hat{\bx} := (\bx - \delta_0 \cdot \bx^*)/(1 - \delta_0)$, then for a certain time $T_0 \ge 1$, the value of $\delta \left( 
\bx^{\pi_{\text{align\&steer}}}(T_0) \right)$ remains equal to $\delta_0$ if and only if $\pi_{\text{steer}}$ fails to steer mass from $\hat{\bx}$ into $\calS^+$ for alignment in the 
preceding $T_0 - 1$ steps. Equivalently, the maximum alignment coefficient of $\hat{\bx}^{\pi_{\text{steer}}}(t)$ is $0$ for \emph{all} $1 \le t \le T_0-1$ (note that $\delta(\hat{\bx})$ is 
$0$ by construction). 

\begin{defi}(Reachability of Optimal Stationary State $\bx^*$) \label{def:reachability}
Fix an initial state $\bx_{\text{init}}$. An optimal stationary state $\bx^*$ of the refined static Problem~\eqref{eq:problem-formulation-stationary-multichain} with $\bx(0) = 
\bx_{\text{init}}$ is called \emph{reachable}, if there exists a feasible and stationary control rule $\pi_{\text{steer}}$, a positive constant $\theta > 0$ and a finite time $T_0 \ge 1$ 
such that the maximum alignment coefficient in Definition~\ref{def:maximum-alignment-coef} satisfies $\delta \left( \bx^{\pi_{\text{steer}}}(T_0) \right) \ge \theta$, with 
$\bx^{\pi_{\text{steer}}}(0) = \bx$ for \emph{all} $\bx \in \Delta$. Otherwise we call $\bx^*$ \emph{unreachable}. 
\end{defi}

From this definition, if $\bx^*$ is unreachable, then for any feasible control $\pi_{\text{steer}}$ and for any $T_0 \ge 1$, there always exists a counterexample $\bx \in \Delta$ along with 
an arm state $s \in \calS^+$, such that the $s$-th coordinate of $\bx^{\pi_{\text{steer}}}(T_0)$ is $0$. This situation can arise due to the non-communicating nature and periodicity issues 
within the single-armed MDP. By definition, a MDP with state space $\calS$ is called \emph{weakly communicating} if $\calS$ can be partitioned into a closed set $\calS^c$ of states in which 
each state is accessible under some deterministic stationary policy from any other state in the set, plus a possibly empty set of states that are transient under every policy. An arm state 
$s \in \calS$ is \emph{aperiodic} under a single-armed policy $\bar{\pi}$ if $\gcd \left\{ n \in \mathbb{N} \mid (\pp^{\bar{\pi}})^n_{ss} > 0 \right\} = 1$, with $\pp^{\bar{\pi}}$ the 
transition matrix of the single-armed Markov chain induced by policy $\bar{\pi}$. 

We now consider the \emph{linear} control defined by $\pi_{\ell}(\bx) := \alpha \cdot \bx$. This is a feasible control according to the definition of $\calU(\bx)$. For $t \ge 0$, by 
plugging into Equation~\eqref{eq:deterministic-transition} this control rule we obtain that $\bx^{\pi_{\ell}}(t) = \bx \cdot (\palpha)^t$, where $\palpha := \alpha \cdot \pp^1 + (1 - 
\alpha) \cdot \pp^0$. Note that $\palpha$ is also the transition matrix of the single-armed Markov chain induced by policy $\bar{\pi}_{\ell}$ "always take action $1$ with probability 
$\alpha$". We argue that a certain communicating and aperiodic condition is sufficient for reachability: 
\begin{lem}(Weakly-Communicating and Aperiodic Single-Armed MDP Implies Reachability) \label{lem:property-of-pi0}
Fix an initial state $\bx_{\text{init}}$. Let $\bx^*$ be an optimal stationary state of the refined static Problem~\eqref{eq:problem-formulation-stationary-multichain} with $\bx(0) = 
\bx_{\text{init}}$, and denote by $\calS^+ = \left\{ s \in \calS \mid x^*_s > 0 \right\}$. If the single-armed MDP in Problem~\eqref{eq:problem-formulation-single-arm} is weakly 
communicating and the set of arm states $\calS^+$ are aperiodic under the single-armed policy $\bar{\pi}_{\ell}$ "always take action $1$ with probability $\alpha$", then $\bx^*$ is 
reachable. 
\end{lem}
\begin{proof}
Take $\pi_{\text{steer}} = \pil$ in Definition~\ref{def:reachability}. Since $\calS^+ \subset \calS^c$, combine with the aperiodicity assumption, there exists $T_0 \ge 1$ such that $\min_{s 
\in \calS, s' \in \calS^+} (\palpha)^{T_0}_{s s'} := p_0 
> 0$. Consequently for all $\bx \in \Delta$, it holds true that 
\begin{equation}  \label{eq:definition-of-theta}
  \delta \left( \bx^{\pi_{\ell}}(T_0) \right) = \delta \left( \bx \cdot (\palpha)^{T_0} \right) \ge \frac{p_0}{\max_{s \in S} x^*_s} := \theta > 0.
\end{equation}
\end{proof}

As a clarification, we point out that the condition of "being aperiodic under $\bar{\pi}_{\ell}$" is less stringent than "being aperiodic under an optimal single-armed policy $\bar{\pi}^*$" 
assumed in \cite{hong2024unichain}. This is because periodicity is a pure graph-theoretic question, and $\bar{\pi}_{\ell}$ leads to the maximum number of directed edges in the connectivity 
graph among all single-armed policies. In addition, we highlight that by refining the notion of reachability in Definition~\ref{def:reachability}, we can also accommodate non-communicating 
cases within our framework. The reason is that any MDP can be partitioned uniquely into communicating classes plus a (possible empty) class of states which are transient under any policy. 
This adaptation necessitates not that all $\bx \in \Delta$ must comply, but only those $\bx$ that possess \emph{the same} positive mass in communicating classes that are common with the 
initial state $\bx_{\text{init}}$. Notably, aperiodicity alone suffices for this refined definition of reachability. Due to space constraints and for clarity of presentation, we opt not to 
delve into this broader generality.

\subsection{Reachability Implies Asymptotic Optimality}  \label{subsec:main-theorem}

We are now ready to state the main result of this paper, which demonstrates that $\piN_{\text{align\&}\ell}$ is asymptotically optimal under the reachability assumption. The key idea is to 
compare the control rule $\pi_{\text{align\&}\ell}$ that maximize the alignment whenever possible with another rule $\pi_{\text{delay\&}\ell}$ that \emph{delay} the alignment. We emphasize 
that, while the function $G^{\pi_{\text{delay\&}\ell}}(\bx)$ for the latter is easier to handle, the delayed alignment control rule is not stationary and depends on the entire history of 
the deterministic state trajectory. As such it cannot be used to construct a stationary policy for the $N$-armed bandit. 

\begin{thm}(Reachability of $\bx^*$ Implies Asymptotic Optimality)  \label{thm:existence-of-good-control-rule}
Fix an initial state $\bx_{\text{init}}$. Suppose Problem~\eqref{eq:problem-formulation-stationary-multichain} with $\bx(0) = \bx_{\text{init}}$ possesses an optimal stationary state 
$\bx^*$ that is reachable as defined in Definition \ref{def:reachability}. Let $\pi_{\text{align\&}\ell}$ represent the align and steer control rule from 
Equation~\eqref{eq:align-and-steer-control}, with the linear steering control rule $\pi_{\ell}(\bx) = \alpha \cdot \bx$. In this case $\pi_{\text{align\&}\ell}$ is effective. Thus, in 
conjunction with Theorem~\ref{thm:cv-rate}, the policy $\piN_{\text{align\&}\ell}$ from Algorithm \ref{algo:maximum-alignment-control} is asymptotically optimal. 
\end{thm}

\begin{proof}
For each finite horizon $T \ge 1$, define 
\begin{equation*}
  G^{\pi}(\bx,T) := \sum_{t=0}^{T-1}  (\bx^{\pi}(t),\bu^{\pi}(t)) - (\bx^*, \bu^*).
\end{equation*}
The function $\bx \mapsto G^{\pi}(\bx,T)$ is a continuous function, provided that the control rule $\pi$ is a continuous map, which holds for $\pi_{\text{align\&}\ell}$. Our strategy for 
proving the continuity of $G^{\pimax}(\bx)$ is to show that the sequence of continuous functions $G^{\pimax}(\bx,T)$ indexed by $T$ converges \emph{uniformly} to $G^{\pimax}(\bx)$ over all 
$\bx \in \Delta$. 

\paragraph{Step One} To ease the notation, let us fix $\bx \in \Delta$ and write $\delta_0 := \delta(\bx)$. Consider another feasible control rule $\pi_{\text{delay\&}\ell}$ that delay the 
alignment in the following sense: We align a fixed amount $\delta_0$ of its mass with $\bx^*$ for the first $T_0$ time steps, where $T_0$ is defined as in Lemma~\ref{lem:property-of-pi0}, 
and constantly apply $\pi_{\ell}$ on the steering part. Formally, the control rule is 
\begin{equation}  \label{eq:lazy-control-rule}
  \pi_{\text{delay\&}\ell}(\bx') := \delta_0 \cdot \bu^* + (1 - \delta_0) \cdot \pi_{\ell} \left(  \frac{\bx' - \delta_0 \cdot \bx^*}{ 1 - \delta_0 } \right),
\end{equation}
for $\bx' = \bx^{\pi_{\text{delay\&}\ell}}(0) = \bx$, and subsequently for $\bx' = \bx^{\pi_{\text{delay\&}\ell}}(1), \dots, \bx^{\pi_{\text{delay\&}\ell}}(T_0-1)$. Denote by $\bx^{(1)} := 
(\bx - \delta_0 \cdot \bx^*)/(1 - \delta_0)$. Then the system trajectory under the control $\pi_{\text{delay\&}\ell}$ in Equation~\eqref{eq:lazy-control-rule} is 
\begin{equation*}
  \bx^{\pi_{\text{delay\&}\ell}}(t) = \delta_0 \cdot \bx^* + (1 - \delta_0) \cdot \bx^{(1)} \cdot (\palpha)^t, \ 0 \le t \le T_0. 
\end{equation*}
As a consequence of Lemma \ref{lem:property-of-pi0}, there exists $\bx^{(2)} \in \Delta$ such that $\bx^{(1)} \cdot (\palpha)^{T_0} = \theta \cdot \bx^* + (1 - \theta) \cdot \bx^{(2)}$, 
with $\theta > 0$ defined in Equation~\eqref{eq:definition-of-theta}. Hence
\begin{align*}
  & \bx^{\pi_{\text{delay\&}\ell}}(T_0)   \\
  = & \ ( 1 - (1 - \delta_0) \cdot (1 - \theta) ) \cdot \bx^* + (1 - \delta_0) \cdot (1 - \theta) \cdot \bx^{(2)}. 
\end{align*}
Now for time steps $T_0+1, T_0 +2, \dots, 2 T_0$ we repeat the same procedure, except that now we start with an alignment of mass $1 - (1 - \delta_0) \cdot (1 - \theta)$ with $\bx^*$. Note 
that 
\begin{equation*}
 \delta_0 < 1 - (1 - \delta_0) \cdot (1 - \theta) \le \delta \big( \bx^{\pi_{\text{delay\&}\ell}}(T_0) \big).
\end{equation*}
With the same reasoning we deduce that there exists $\bx^{(3)} \in \Delta$ such that 
\begin{align*}
  & \bx^{\pi_{\text{delay\&}\ell}}(2 T_0) \\
  = &  \left( 1 - (1 - \delta_0) \cdot (1 - \theta)^2 \right) \cdot \bx^* + (1 - \delta_0) \cdot (1 - \theta)^2 \cdot \bx^{(3)}.
\end{align*}
By a straightforward induction, we infer that for all integer $k \ge 1$, there exists $\bx^{(k+1)} \in \Delta$ such that 
\begin{align*}
  & \bx^{\pi_{\text{delay\&}\ell}}(k \cdot T_0)  \\
  = & \left( 1 - (1 - \delta_0) \cdot (1 - \theta)^k \right) \cdot \bx^*  + (1 - \delta_0) \cdot (1 - \theta)^k \cdot \bx^{(k+1)},
\end{align*}
and consequently 
\begin{equation}  \label{eq:argument-1-proof-good}
  \delta \big( \bx^{\pi_{\text{delay\&}\ell}}(t) \big) \ge 1 - (1 - \delta_0) \cdot (1 - \theta)^k
\end{equation}
for $k \cdot T_0 \le t < (k+1) \cdot T_0$.

\paragraph{Step Two} Our next observation is that for $t \ge 0$ it holds true that
\begin{equation}  \label{eq:argument-2-proof-good}
  \delta \left( \bx^{\pi_{\text{delay\&}\ell}}(t) \right) \le \delta \left( \bx^{\pi_{\text{align\&}\ell}}(t) \right).
\end{equation}
Indeed, delaying the alignment of a certain mass with $\bx^*$ invariably leads to a reduced maximum alignment coefficient in the future, compared to aligning this mass with $\bx^*$ at an 
earlier time. To illustrate, consider $0 < \delta_1 < \delta_0$ and express $\bx$ as follows: 
\begin{equation*}
  \bx = \delta_1 \cdot \bx^* + (\delta_0 - \delta_1) \cdot \bx^* + (1-\delta_0) \cdot \bx^{(1)}.
\end{equation*}
Should the mass amounting to $\delta_0 - \delta_1$ remain unaligned, it will be subjected to the linear control rule $\pil$, along with $\bx^{(1)}$. A portion of this mass may eventually 
become aligned at a later time. However, irrespective of the specific decisions made concerning $(\delta_0 - \delta_1) \cdot \bx^*$, what happens on $(1-\delta_0) \cdot \bx^{(1)}$ remains 
unchanged. Consequently, at any future point, the maximum alignment coefficient achieved by aligning $\delta_0$ at an initial step is always larger than that of aligning $\delta_1$ at the 
same juncture. 

\paragraph{Step Three} To conclude the proof, let $\varepsilon > 0$ be fixed. We define a finite horizon $T(\varepsilon) = k(\varepsilon) \cdot T_0$, where the value of $k(\varepsilon) \in 
\mathbb{N}$ will be chosen later. Then
\begin{align*}
  & \norme{G^{\pimax}(\bx) - G^{\pimax}(\bx,T(\varepsilon))} \\
 \le & \sum_{t=T(\varepsilon)}^{\infty} \left( 1 - \delta(\bx^{\pimax}(t)) \right)  \\
  & \qquad \cdot \left( \norme{ \widetilde{\bx}(t) - \bx^*} +  \norme{ \pil \left( \widetilde{\bx}(t) \right) - \bu^*} \right) \\
  & \left( \text{We abbreviate $ \frac{\bx^{\pimax}(t)-\delta(\bx^{\pimax}(t)) \cdot \bx^*}{1-\delta(\bx^{\pimax}(t))}$ as $\widetilde{\bx}(t)$}.  \right) \\
 \le & \ 2(1+\alpha) \cdot \sum_{t=T(\varepsilon)}^{\infty} \left( 1 - \delta(\bx^{\pidelay}(t)) \right)  \text{ (By Equation~\eqref{eq:argument-2-proof-good}) } \\
 \le & \ 2(1+\alpha)(1-\delta_0) T_0 \cdot \sum_{k=k(\varepsilon)}^{\infty} (1-\theta)^k \text{ (By Equation~\eqref{eq:argument-1-proof-good}) } \\
 = & \ 2(1+\alpha)(1-\delta_0) (1-\theta)^{k(\varepsilon)} \cdot \frac{T_0}{\theta}. 
\end{align*}
It suffices to choose $k(\varepsilon) := \left\lceil \log_{1-\theta} \frac{\varepsilon \cdot \theta}{ 2 (1+\alpha)(1-\delta_0) \cdot T_0} \right\rceil$ to deduce that 
$\norme{G^{\pimax}(\bx) - G^{\pimax}(\bx,T(\varepsilon))} \le \varepsilon$ for all $\bx \in \Delta$. Since the choice of $\varepsilon > 0$ is arbitrary, we conclude by the uniform 
convergence theorem that $G^{\pimax}(\bx)$ is a continuous function defined over all $\bx \in \Delta$. 
\end{proof}

\subsection{Discussion on the Reachability Assumption}  \label{subsec:discussion-reachability}

We first remark that it is crucial to utilize the refined static Problem~\eqref{eq:problem-formulation-stationary-multichain} for computing the optimal stationary point in multichain 
models. For instance, consider \( \pp^0 = \pp^1 = \begin{pmatrix} 1 & 0 \\ 0 & 1 \end{pmatrix} \), $\br^0 = (1,0)$, $\br^1 = (0,1) $, and with any $0 < \alpha < 1$. This model is 
non-communicating. If we solve the conventional static problem described after Problem~\eqref{eq:problem-formulation-stationary-multichain}, then the optimal stationary state is $\bx^* = 
(\alpha,1-\alpha)$ and is unreachable unless $\bx_{\text{init}} = (\alpha,1-\alpha)$; while if we take the initial state $\bx_{\text{init}}$ into account and solve the refined static 
Problem~\eqref{eq:problem-formulation-stationary-multichain} with $\bx(0) = \bx_{\text{init}}$, the optimal stationary state is $\bx^* = \bx_{\text{init}}$ itself and becomes reachable. 

We now compare previous methods based on the single-armed MDP with our approach in a multichain model. Set \( \pp^0 = 
\begin{pmatrix} 1 & 0 & 0 & 0 \\ 0.9 & 0 & 0.1 & 0 \\ 0 & 0 & 0.9 & 0.1 \\ 0 & 0 & 0.1 & 0.9 \end{pmatrix} \), \( \pp^1 = 
\begin{pmatrix} 0.9 & 0.1 & 0 & 0 \\ 0.1 & 0.9 & 0 & 0 \\ 0 & 0 & 1 & 0 \\ 0.1 & 0 & 0.9 & 0 \end{pmatrix} \), $\br^0 = (0,0,1,1)$, $\br^1 = (1,1,0,0)$ and $\alpha = 0.5$. The initial state
is $\bx_{\text{init}} = (0.4,0,0.6,0)$, with the corresponding optimal stationary point $\bx^* = (0.25,0.25,0.25,0.25), \bu^* = (0.25,0.25,0,0)$. It turns out that a \emph{stationary} 
optimal single-armed policy does not exist for this problem, due to the state-action frequency constraint. As highlighted in \cite{hordijk1984constrained}, exploring the broader class of 
Markovian policies is necessary to find an optimal solution. Consequently, approaches based on stationary optimal single-armed policies, as seen in \cite{hong2023restless,hong2024unichain}, 
are inadequate. Similarly, for the fluid priority policy \cite{Ve2016.6}, any priority "permutation of (1,2) > permutation of (3,4)" falls within this category. However, prioritizing "4 > 
3" for this particular initial state $\bx_{\text{init}} = (0.4,0,0.6,0)$ is essential for asymptotic optimality; without it, the 0.6 portion of arms in states $\{3,4\}$ cannot transition to 
states $\{1,2\}$. Policies based on the single-armed MDP typically lack the capability to make such critical distinctions. In contrast, since the model is communicating and aperiodic in the 
sense of Lemma~\ref{lem:property-of-pi0}, we deduce that $\bx^*$ is reachable and consequently by Theorem~\ref{thm:existence-of-good-control-rule}, the policy $\piN_{\text{align\&}\ell}$ is 
asymptotically optimal.

\section{Numerical Experiments}  \label{sec:numerical-example}

It can be seen that the simple linear control $\pil$ that has played a key role in Theorem~\ref{thm:existence-of-good-control-rule} may not be the best candidate for the task of steering.  
Ideally the steering control $\pi_{\text{steer}}$ should be designed to take any vector $\bx$ towards $\bx^*$ in the most reward-efficient manner. Motivated by the Model Predictive Control 
(MPC), a such control can be constructed by solving a finite look-ahead window $T_w$ version of Problem~\eqref{eq:problem-CEC-infinite-horizon}, which is a linear program 
(\cite{gast2023linear}), followed by adopting the first control from this solution, see e.g. \cite[Section 3]{damm2014exponential}. We refer to this policy using MPC steering strategy as 
$\piN_{\text{align\&MPC}}$. In this section, we set a look-ahead window of $T_w = 100$, noting that the MPC appears to stabilize at $50$ steps ahead on the examples encountered 
\cite{damm2014exponential}. Simulations are conducted over a horizon of $T=10000$. As highlighted in Section~\ref{subsec:discussion-reachability}, previous methods generally fail with 
multichain models. Hence our focus in this section lies on unichain models and the performance differences for finite $N$.

We first consider an example with three states that has been studied in \cite{gast2023exponential,hong2023restless}. A noticeable feature of this example is that the Whittle index policy, 
which is actually the best-performed priority policy among all possible priorities, is not asymptotically optimal, as can be inferred from Figure~\ref{fig:example-1}. We also plot 
$\delta(\bX(t))$ over a sample run for the three policies with $N=1000$ once each $5$ time steps for the first $200$ time steps. The oscillation of $\pi_{\text{priority}}$ presented here is 
caused 
\begin{wrapfigure}[17]{l}{0.24\textwidth} 
  \centering
  \includegraphics[width=0.24\textwidth]{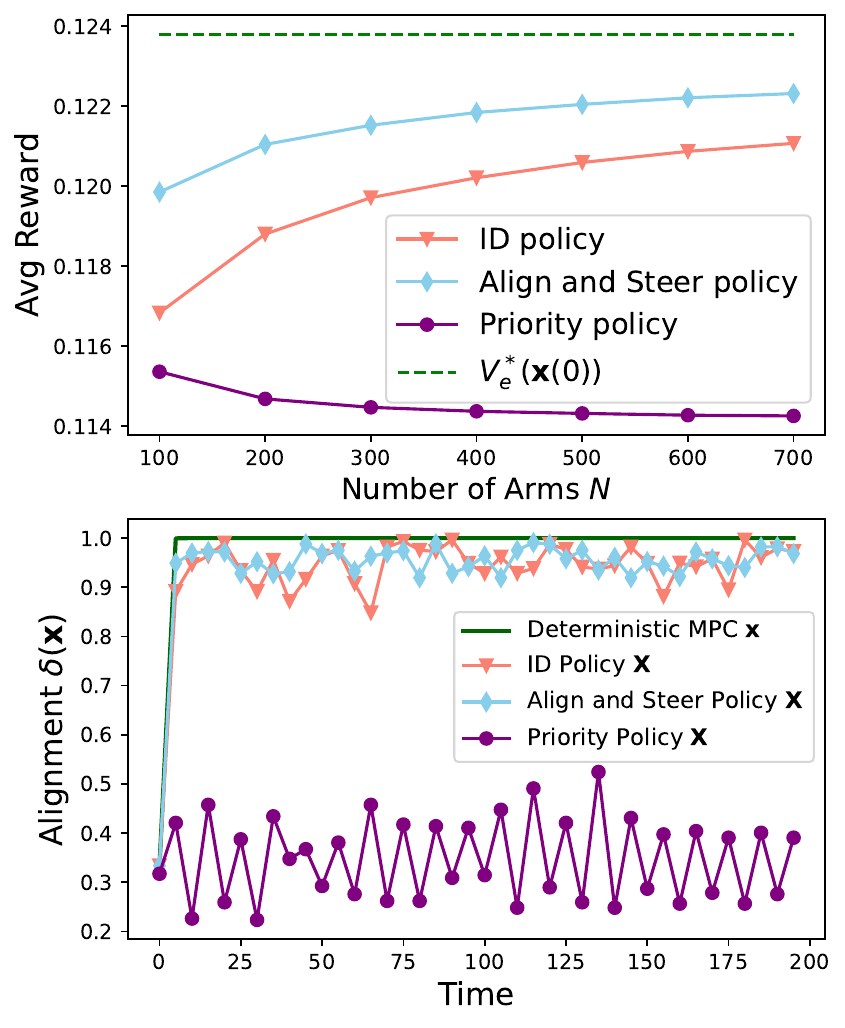}
  \caption{\small An example where no priority policy is asymptotically optimal.}
  \label{fig:example-1}
\end{wrapfigure} 
by the fact that its dynamics is attracted to a period-$2$ cycle. The ID policy is introduced in \cite{hong2024unichain}. Both the ID policy and the align and steer policy proposed in the  
current paper are asymptotically optimal, but the later performs better. We believe that this is because the closed-loop MPC  is constantly driving the steering part to align with $\bx^*$ 
in the most reward-efficient way. We note that it is also possible to plot the alignment of the control variable with $\bu^*$, which exhibits similar behavior compared to the alignment of 
the state variable. 

We next consider an example with eight states that has been proposed in \cite{hong2023restless}. A noteworthy feature of this example is that there are a total number of $36$ priority 
\begin{wrapfigure}[17]{r}{0.23\textwidth} 
  \centering
  \includegraphics[width=0.23\textwidth]{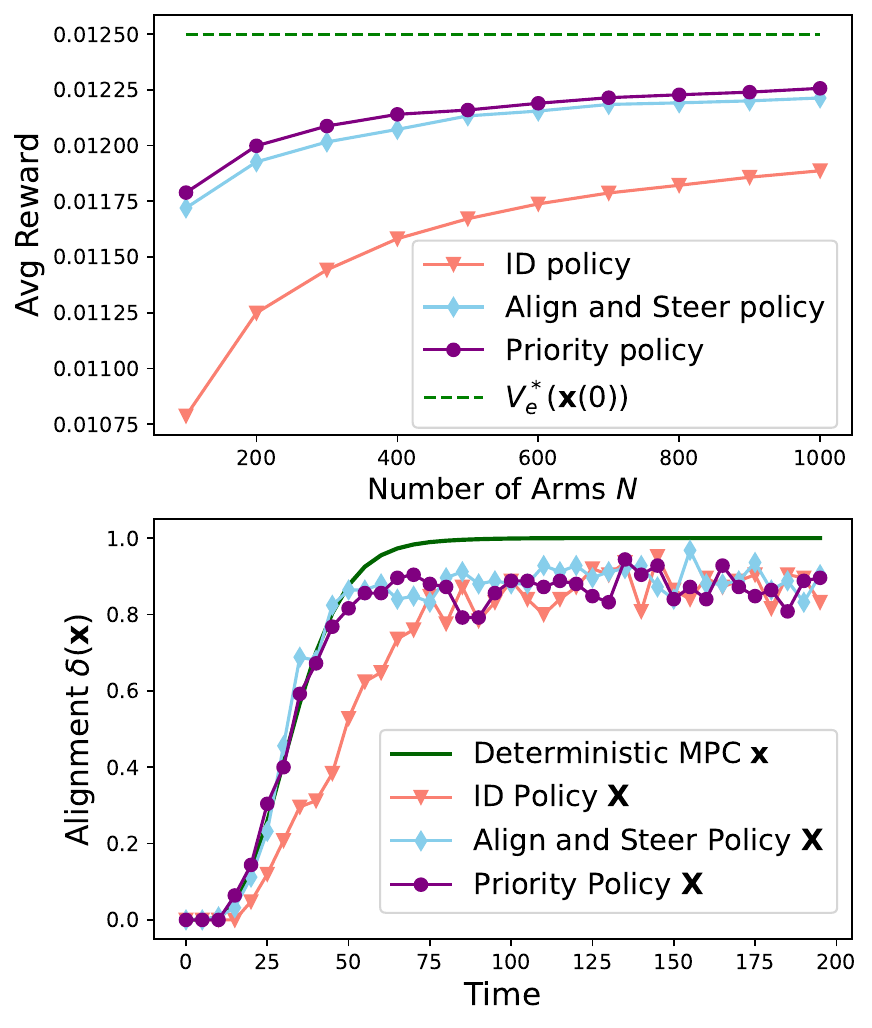}
  \caption{\small An example where certain priority policies perform slightly better than $\piN_{\text{align\&MPC}}$.}
  \label{fig:example-2}
\end{wrapfigure}
policies that are asymptotically optimal: "any permutation of (2,3,4) > 1 > 5 > any permutation of (6,7,8)" are all observed to be asymptotically optimal policies. The performance of these 
priority policies are actually slightly superior than $\piN_{\text{align\&MPC}}$, as can be seen from Figure~\ref{fig:example-2}. We visualize that the steering of both policies align with 
the deterministic MPC, while the steering of the ID policy is different. 

From the numerical analysis presented in this section, it is evident that $\piN_{\text{align\&MPC}}$ consistently delivers outstanding performance. However, given its computational 
efficiency and simplicity, a priority policy such as the LP-index policy from \cite{gast2023linear} should be considered at first hand. Only in instances where the global attractor property 
does not seem to be fulfilled by these priority policies should we consider resorting to the ID policy or $\piN_{\text{align\&MPC}}$. The former necessitates sampling of arm actions and 
their rectification at every decision epoch, which requires at least $\calO(N)$ time; the latter entails solving a linear program being the transient problem of 
\eqref{eq:problem-CEC-infinite-horizon} with a horizon of $T_w$ at each time step.



\section{Extension and Conclusion}  \label{sec:extension-and-conclusion}

\subsection{On the Generalization to Weakly-Coupled MDPs}   

The weakly-coupled MDP is a substantial and natural extension of the restless multi-armed bandit, characterized by each (homogeneous) arm having multiple actions (i.e., $\abs{\calA} 
> 2$) and the imposition of multiple budget constraints on the bandit. To ensure problem feasibility, it is assumed that there exists an action $0$ that does not consume any resources, as 
for the RB. This topic has been explored in a series of studies, including \cite{10.1287/opre.1070.0445,gast2022lp,brown2023fluid},  within both finite-horizon and discounted 
infinite-horizon frameworks, but not within the undiscounted setting addressed in the current paper. A notable aspect of the optimal-control approach adopted in this paper is its 
straightforward applicability to weakly-coupled MDPs with minimal additional effort required, thus filling an important gap in existing research. 

Indeed, the CEC problem for weakly-coupled MDPs can be generally expressed as 
\begin{alignat}{2}
    & \underset{\pi}{\text{max}} \quad && V_{\pi}(\mathbf{x}(0)) := \limT R(\mathbf{x}(t),\mathbf{u}(t)) \label{eq:problem-CEC-infinite-horizon-WCMDP} \\
    & \text{s.t.} \quad && \mathbf{x}(t+1) =  \phi(\mathbf{x}(t),\mathbf{u}(t)), \notag \\
    & && f(\mathbf{x}(t),\mathbf{u}(t)) = \mathbf{0}, \quad g(\mathbf{x}(t),\mathbf{u}(t)) \le \mathbf{0}, \quad \forall t \ge 0. \notag 
\end{alignat}
Here, the control variable $\bu$ is a vector of size $\abs{\calS} \times \abs{\calA}$; $R(\bx,\bu)$ represents a general linear function denoting the instant-reward; $\phi(\bx,\bu)$ is a 
linear function describing the expected Markovian transition as in Equation~\eqref{eq:deterministic-transition}; and $f(\bx,\bu)$, $g(\bx,\bu)$ are linear functions related to budget and 
problem structure constraints. Leveraging results from \cite{altman1995linear, hordijk1984constrained}, which focus on undiscounted average-reward multichain MDPs with linear state-action 
frequency constraints, we can deduce as in the RB case the relationships in Equation~\eqref{eq:relation-of-problem-values} for the various optimization problems. Consequently, the approach 
outlined in this paper can be applied to this more general context as well. The crucial aspect that facilitates this extension is the \emph{linearity} of the CEC optimal-control 
Problem~\eqref{eq:problem-CEC-infinite-horizon-WCMDP}.

\subsection{Conclusion} \label{subsec:conclusion-open-questions}

In this work, we have introduced an optimal-control framework for the undiscounted infinite-horizon $N$-armed RB problem, focusing on relaxing hard constraints to expected trajectory 
constraints at each time step, unlike traditional methods that average these constraints over time. This approach, balancing complexity between overly-simplified single-armed MDPs and 
intractable N-armed RB problems, allows us to derive asymptotically optimal policies by steering the system towards an optimal stationary state within a deterministic framework. Future 
research directions include: 

\subsubsection{The Lipschitz-Continuity of $G^\pi(\bx)$} Under the generality considered in this work, the possibility of ensuring Lipschitz-continuity for $G^\pi(\bx)$, which implies 
$\calON$ convergence rates of the induced policy $\piN$ \cite{gast2023exponential}, remains open. The applicability of Lyapunov-function-based proof techniques from single-armed MDPs to 
multichain scenarios \cite{hong2024unichain}, or finding counter-examples would significantly advance our understanding.

\subsubsection{The Exponential Turnpike Property and Choice of Lookahead Window} Investigating the exponential turnpike property's role (\cite{damm2014exponential}) in determining the finite 
lookahead window $T_w$ for MPC controls could greatly impact the efficiency of applying our framework in practice. This exploration could also yield crucial insights into the dynamics of RB 
optimal-control problems.

\bibliographystyle{plain} 
\bibliography{reference}

\section{Proof of A Technical Lemma}  \label{appendix:proof-lemmas}

\begin{lem}(CEC is an Upper Bound)  \label{lem:CEC-is-upperbound}
Let $V_{\piN}(\bx(0))$ (resp. $V_{\pi}(\bx(0))$) denote the optimal value of Problem~\eqref{eq:problem-formulation-population-representation-infinite-horizon} (resp. 
Problem~\eqref{eq:problem-CEC-infinite-horizon}). Then it holds true that $V_{\piN}(\bx(0)) \le V_{\pi}(\bx(0))$. 
\end{lem}

\begin{proof}
We show by induction on the horizon $T$. Denote by $V_{\piN}(\bx(0),T)$ the value multiplied by $T$ of the finite-horizon-$T$ version of 
Problem~\eqref{eq:problem-formulation-population-representation-infinite-horizon}, with a similar notation $V_{\pi}(\bx(0),T)$ for Problem~\eqref{eq:problem-CEC-infinite-horizon}). Suppose 
that $V_{\piN}(\bX,T-1) \le V_{\pi}(\bX,T-1)$ holds for all $\bX \in \Delta$, and denote by the instant-reward function $R(\bX,\bU) = (\bX-\bU) \cdot (\br^0)^\top + \bU \cdot (\br^1)^\top$. 
Then 
  \begin{align*}
    & V_{\piN}(\bX,T)  \\
    = &  \max_{\bU \in \calUN(\bX)} R(\bX,\bU) \\
    & \qquad + \expect{V_{\piN} \left( \phi(\bX,\bU) + \calE(\bX,\bU), T-1 \right) \mid \bX, \bU } \\
    \le & \max_{\bU \in \calUN(\bX)} R(\bX,\bU) \\
    & \qquad +  V_{\piN} \left( \expect{ \phi(\bX,\bU) + \calE(\bX,\bU) \mid \bX, \bU }, T-1 \right) \\
    & \qquad \text{(By concavity of the value function.)} \\
    = & \max_{\bU \in \calUN(\bX)} R(\bX,\bU) +  V_{\piN} \left(  \phi(\bX,\bU), T-1 \right)  \\
     & \qquad \text{(By Lemma~\ref{lem:Markovian-transition-analysis}.)} \\
    \le & \max_{\bU \in \calU(\bX)} R(\bX,\bU) +  V_{\piN} \left(  \phi(\bX,\bU), T-1 \right)  \\
     & \qquad \text{(Since $\calUN(\bx) \subset \calU(\bx)$.)} \\
    \le & \max_{\bU \in \calU(\bX)} R(\bX,\bU) +  V_{\pi} \left(  \phi(\bX,\bU), T-1 \right)  \\
    & \qquad \text{(By induction hypothesis.)} \\
    = & V_{\pi}(\bX,T).
  \end{align*}  
To conclude the proof it suffices to take the limit $T \rightarrow \infty$ in the relation $V_{\piN}(\bX,T)/T \le V_{\pi}(\bX,T)/T$. 
\end{proof}

\section{Proof of Theorem~\ref{thm:cv-rate}}  \label{appendix:proof-theorem-1}

\textbf{Theorem~\ref{thm:cv-rate} restated:} Fix an initial state $\bx_{\text{init}}$. For a stationary control rule $\pi$ of Problem~\eqref{eq:problem-CEC-infinite-horizon} with an optimal 
stationary point $(\bx^*, \bu^*)$ and optimal value $V^*_e(\bx(0))$ defined in Problem~\eqref{eq:problem-formulation-stationary-multichain} with $\bx(0) = \bx_{\text{init}}$, consider the 
function $G^{\pi}(\bx)$ defined in Equation~\eqref{eq:bias}. If $G^{\pi}(\bx)$ is a continuous function over $\bx \in \Delta$ (under the $\mathcal{L}^1$-norm). Then the induced stationary 
policy $\piN$ for Problem~\eqref{eq:problem-formulation-population-representation-infinite-horizon} in Definition~\ref{def:rounding-policy} is asymptotically optimal: $\limN 
V_{\piN}(\bx(0)) = V^*_e(\bx(0))$. 

\begin{proof}
Given $\bX \in \DeltaN$ and a policy $\piN$, we denote $\bX^{\piN}(t)$ as the random variable representing the state of the stochastic system at time $t$ under policy $\piN$. That is 
$\bX^{\piN}(0) = \bX$ and for $t \ge 0$, 
\begin{equation*}
  \bX^{\piN}(t+1) \overset{d}{=} \phi(\bX^{\piN}(t), \bU^{\piN}(t)) + \mathcal{E}(\bX^{\piN}(t), \bU^{\piN}(t)).
\end{equation*}
This should be contrasted with the notation $\bx^{\pi}(t)$ that concerns deterministic trajectory, which is defined as $\bx^{\pi}(0) = \bx$ and for $t \ge 0$, 
\begin{equation*}
  \bx^{\pi}(t+1) = \phi(\bx^{\pi}(t), \bu^{\pi}(t)).
\end{equation*}
Denote by $\bX^{\piN}(\infty)$ the steady state of the induced Markov chain of the $N$-armed bandit under the stationary policy $\piN$, and by $\bU^{\piN}(\infty)$ the corresponding steady 
state of the controls. By stationarity, the random variable $( \bX^{\piN}(\infty) )^{\piN}(1)$ has the same law as $\bX^{\piN}(\infty)$. 

Since 
\begin{align*}
 & V_e^*(\bx(0)) - V_{\piN}(\bx(0)) \\
 = & \ \bx^* \cdot (\br^0)^\top + \bu^* \cdot (\br^1-\br^0)^\top \\
 & \qquad - \expect{ \bX^{\piN}(\infty) \cdot (\br^0)^\top + \bU^{\piN}(\infty) \cdot (\br^1-\br^0)^\top } \\
 \le & \norme{ \expect{ \left( \bX^{\piN}(\infty), \bU^{\piN}(\infty) \right) } - (\bx^*,\bu^*) } \cdot 2 r_{\text{max}},
\end{align*}
where $r_{\text{max}} := \max_{s,a} \abs{r_s^a}$. In what follows it suffices to bound the $\calL^1$-norm on the right hand side of the above formula. 

For \emph{any} $\bx = \bx^{\pi}(0)$, by the definition of $G^{\pi}(\bx)$, we have 
\begin{align}
  & G^{\pi}(\bx^{\pi}(0)) - G^{\pi}(\bx^{\pi}(1)) \nonumber \\
  = & \sum_{t=0}^{\infty} (\bx^{\pi}(t), \bu^{\pi}(t)) - (\bx^*,\bu^*) \nonumber \\
  & \qquad + \sum_{t=0}^{\infty} (\bx^{\pi}(t+1), \bu^{\pi}(t+1)) - (\bx^*,\bu^*) \nonumber \\
  = & \sum_{t=0}^{\infty} (\bx^{\pi}(t), \bu^{\pi}(t)) - (\bx^*,\bu^*) \nonumber \\
  & \qquad + \sum_{t=1}^{\infty} (\bx^{\pi}(t), \bu^{\pi}(t)) - (\bx^*,\bu^*) \nonumber \\
  = & (\bx^{\pi}(0), \bu^{\pi}(0)) - (\bx^*,\bu^*). \label{eq:relation-on-G-1}
\end{align}
From which we infer that 
\begin{align}
  &  \expect{ \left( \bX^{\piN}(\infty), \bU^{\piN}(\infty) \right) }  - (\bx^*,\bu^*) \nonumber \\
  = & \ \expect{ \expect{G^{\pi} \left( \bX^\pi(0) \right) - G^{\pi} \left( \bX^{\pi}(1) \right) \mid \bX^{\piN}(\infty) = \bX^{\pi}(0)} } \nonumber \\
  = & \ \mathbb{E} \bigg[ G^{\pi} \left( \bX^{\piN}(\infty) \right) - G^{\pi} \left( \left( \bX^{\piN}(\infty) \right)^{\pi}(1) \right) \bigg]  \nonumber \\
  = & \ \mathbb{E} \bigg[ G^{\pi} \left( \left( \bX^{\piN}(\infty) \right)^{\piN}(1) \right) - G^{\pi} \left( \left( \bX^{\piN}(\infty) \right)^{\pi}(1) \right) \bigg]  \nonumber \\
  = & \ \expect{\zeta \left(  \bX^{\piN}(\infty) \right) }, \label{eq:relation-on-G-2}
\end{align}
where for $\bX \in \DeltaN$, we define
\begin{align}  
  & \zeta(\bX) := \expect{G^{\pi} \left( \phi(\bX,\piN(\bX)) + \calE(\bX,\piN(\bX)) \right)} \nonumber \\
  & \qquad - G^{\pi} \left( \phi(\bX,\pi(\bX)) \right).  \label{eq:zeta-x}
\end{align}
We shall show below that 
\begin{equation}\label{eq:auxilary-result}
  \limN \max_{\bX \in \DeltaN} \norme{\zeta(\bX)} = 0
\end{equation}
And subsequently from Equation~\eqref{eq:relation-on-G-2} we deduce that 
\begin{align*}
  & \limN \norme{ \expect{ \left( \bX^{\piN}(\infty), \bU^{\piN}(\infty) \right) }  - (\bx^*,\bu^*) } \\
  & \qquad \le  \limN \max_{\bX \in \DeltaN} \norme{\zeta(\bX)} = 0.
\end{align*}
So we conclude that $\limN V_{\piN}(\bx(0)) = V^*_e(\bx(0))$.

We now use the uniform continuity of the function $G^\pi(\bx)$ to prove the claim in Equation~\eqref{eq:auxilary-result}. Denote by $L_{\phi} > 0$ the Lipschitz constant of the linear map 
$\phi$, and by $\Phi(\bX) := \phi(\bX,\pi(\bX))$. Define 
\begin{equation} \label{eq:CLT-and-rounding}
  \widetilde{\calE}(\bX) :\overset{d}{=} \phi(\bX,\piN(\bX)) + \calE(\bX,\piN(\bX)) - \Phi(\bX).
\end{equation}
By our construction of the policy $\piN$ in Definition~\ref{def:rounding-policy}, we deduce that  
\begin{align*}  
  & \expect{ \norme{ \widetilde{\calE}(\bX) } \mid \bX } \le  \frac{S L_{\phi}}{N} + \frac{\sqrt{S}}{\sqrt{N}};  \\ 
  & \proba{ \norme{\widetilde{\calE}(\bX)} \ge \xi \mid \bX} \le 2S \cdot e^{-N (\xi/2)^2/S^2}.
\end{align*}
Fix $\varepsilon > 0$ arbitrary, there exists $\xi > 0$ such that 
\begin{equation*}
  \norme{\widetilde{\calE}(\bX)} \le \xi \Rightarrow \norme{ G^\pi \left( \Phi(\bX) + \widetilde{\calE}(\bX) \right) - G^\pi \left( \Phi(\bX) \right) } \le \varepsilon.
\end{equation*}
Denote by $\overline{G} := \max_{\bX \in \Delta} G^\pi(\bx) < \infty$, we deduce that 
\begin{align*}
  & \max_{\bX \in \DeltaN} \norme{\zeta(\bX)} \\
 \le & \max_{\bX \in \DeltaN} \expect{ \norme{ G^\pi \left( \Phi(\bX) + \widetilde{\calE}(\bX) \right) - G^\pi \left( \Phi(\bX) \right) } \mid \bX } \\
 \le & \ 2 \overline{G} \cdot \proba{ \norme{\widetilde{\calE}(\bX)} \ge \xi \mid \bX} + \varepsilon \\
 \le & \ 2S \cdot e^{-N (\xi/2)^2/S^2} + \varepsilon \le 2 \varepsilon 
\end{align*}
for $N$ sufficiently large. Since the choice of $\varepsilon > 0$ is arbitrary, we infer that $\limN \max_{\bX \in \DeltaN} \norme{\zeta(\bX)} = 0$. 
\end{proof}

\end{document}